\documentclass[a4paper,11pt,psamsfonts]{amsart}

\usepackage{amsmath,amsfonts,amssymb}
\usepackage{xypic}
\usepackage[english]{babel}
\usepackage[utf8x]{inputenc}
\usepackage{graphicx}
\usepackage{multirow}
\usepackage{array}
\usepackage{booktabs}
\usepackage{url}
\usepackage{fullpage}
\usepackage{epstopdf}
\usepackage{ifpdf}

\ifpdf
\fi

\newtheorem{thm}{Theorem}
\newtheorem{prop}{Proposition}[section]
\newtheorem{cor}[prop]{Corollary}

\theoremstyle{definition}

\newtheorem{rem}[prop]{Remark}
\newtheorem{ex}[prop]{Example}

\newcommand{\N}{\mathbb{N}}
\newcommand{\Z}{\mathbb{Z}}

\newcommand{\C}{\mathbb{C}}

\newcommand{\st}{\;:\;}

\newcommand{\GL}{\mathrm{GL}}

\DeclareMathOperator{\im}{i}
\DeclareMathOperator{\imm}{im}

\DeclareMathOperator{\de}{d}

\newcommand{\del}{\partial}
\newcommand{\delbar}{\overline{\del}}

\title[On the $\del\delbar$-Lemma and Bott-Chern cohomology]
{On the $\del\delbar$-Lemma and Bott-Chern cohomology}
\author{Daniele Angella}
\address[Daniele Angella]{Dipartimento di Matematica ``Leonida Tonelli''\\
Universit\`{a} di Pisa \\
Largo Bruno Pontecorvo 5, 56127\\
Pisa, Italy}
\email{angella@mail.dm.unipi.it}
\author{Adriano Tomassini}
\address[Adriano Tomassini]{Dipartimento Di Matematica\\
Universit\`{a} di Parma \\
Parco Area delle Scienze 53/A, 43124 \\
Parma, Italy}
\email{adriano.tomassini@unipr.it}

\thanks{This work was supported
by GNSAGA of INdAM}
\begin{document}

\vspace{-2cm}
\begin{minipage}[l]{10cm}
{\sffamily
  D. Angella, A. Tomassini, On the $\partial\overline\partial$-Lemma and Bott-Chern cohomology,
  {\em Invent. Math.} \textbf{192} (2013), no.~1, 71--81.

\smallskip

  \begin{flushright}\begin{footnotesize}
  (The original publication is available at \url{www.springerlink.com}.)
  \end{footnotesize}\end{flushright}
}
\end{minipage}
\vspace{2cm}

\begin{abstract}
On a compact complex manifold $X$, we prove a Fr\"olicher-type inequality for Bott-Chern cohomology and we show that the equality holds if and only if $X$ satisfies the $\del\delbar$-Lemma.
\end{abstract}

\maketitle

\section*{Introduction}
An important cohomological invariant for compact complex manifolds is provided by the Dolbeault cohomology. While in the compact K\"ahler case the Hodge decomposition theorem states that the Dolbeault cohomology groups give a decomposition of the de Rham cohomology, this holds no more true, in general, for non-K\"ahler manifolds.\\
Nevertheless, on a compact complex manifold $X$, the Hodge-Fr\"olicher spectral sequence $E^{\bullet,\bullet}_1\simeq H^{\bullet,\bullet}_{\delbar}(X)\Rightarrow H^\bullet_{dR}(X;\C)$, see \cite{frolicher}, links Dolbeault cohomology to de Rham cohomology, giving in particular the {\em Fr\"olicher inequality}:
$$ \text{for every }k\in\N\;,\qquad \sum_{p+q=k}\dim_\C H^{p,q}_{\delbar}(X) \;\geq\; \dim_\C H^{k}_{dR}(X;\C) \;.$$

Other important tools to study the geometry of compact complex (especially, non-K\"ahler) manifolds are the {\em Bott-Chern} and {\em Aeppli} cohomologies, that is,
$$ H^{\bullet,\bullet}_{BC}(X) \;:=\; \frac{\ker\del\cap\ker\delbar}{\imm\del\delbar} \qquad \text{ and }\qquad H^{\bullet,\bullet}_{A}(X) \;:=\; \frac{\ker\del\delbar}{\imm\del+\imm\delbar} \;.$$
While they coincide with the Dolbeault cohomology in the K\"ahler case, they supply further informations on the complex structure of a non-K\"ahler manifold. Fixing a Hermitian metric, as a consequence of the Hodge theory, one has that they are finite-dimensional $\C$-vector spaces and that there is an isomorphism between Bott-Chern and Aeppli cohomologies.\\
These cohomology groups have been recently studied by J.-M. Bismut in the context of Chern characters (see \cite{bismut}) and by L.-S. Tseng and S.-T. Yau in the framework of generalized geometry and type II string theory (see \cite{tseng-yau}).

A very special condition in complex geometry from the cohomological point of view is provided by the {\em $\del\delbar$-Lemma}: namely, a compact complex manifold is said to satisfy the $\del\delbar$-Lemma if every $\del$-closed, $\delbar$-closed, $\de$-exact complex form is $\del\delbar$-exact. For example, compact K\"ahler manifolds or, more in general, manifolds in {\em class $\mathcal{C}$ of Fujiki} (that is, compact complex manifolds admitting a proper K\"ahler modification) satisfy the $\del\delbar$-Lemma (see the paper by P. Deligne, Ph. Griffiths, J. Morgan and D. Sullivan, \cite{deligne-griffiths-morgan-sullivan}).

\medskip

In this note, we study the relations between Bott-Chern and Aeppli cohomologies and $\del\delbar$-Lemma.

More precisely, we prove the following result, stating a Fr\"olicher-type inequality also for Bott-Chern and Aeppli cohomologies and giving a characterization of the validity of the $\del\delbar$-Lemma just in terms of the dimensions of $H^{\bullet,\bullet}_{BC}(X;\C)$.

\smallskip
\noindent {\bfseries Theorem} (see Theorem \ref{thm:frol-bc} and Theorem \ref{thm:caratterizzazione-bc-numbers}){\bfseries.}\
{\itshape
 Let $X$ be a compact complex manifold. Then, for every $k\in\N$, the following inequality holds:
\begin{equation*}\tag{\ref{eq:frol-bc-2}}
\sum_{p+q=k} \left( \dim_\C H^{p,q}_{BC}(X)+\dim_\C H^{p,q}_{A}(X) \right) \;\geq\; 2\,\dim_\C H^k_{dR}(X;\C) \;.
\end{equation*}
Moreover, the equality in \eqref{eq:frol-bc-2} holds for every $k\in\N$ if and only if $X$ satisfies the $\del\delbar$-Lemma.
}
\smallskip

As a consequence of the previous theorem, we obtain another proof of the stability of the $\del\delbar$-Lemma under small deformations of the complex structure (see \cite{voisin, wu}), see Corollary \ref{cor:stab}.

\section{Preliminaries and notations}\label{sec:preliminaries}
In this section, we recall some notions and results we need in the sequel.\\
Let $X$ be a compact complex manifold of complex dimension $n$.

\medskip

The {\em Bott-Chern cohomology} of $X$ is the bi-graded algebra
$$ H^{\bullet,\bullet}_{BC}(X) \;:=\; \frac{\ker\del\cap\ker\delbar}{\imm\del\delbar} $$
and the {\em Aeppli cohomology} of $X$ is the bi-graded $H^{\bullet,\bullet}_{BC}(X)$-module
$$ H^{\bullet,\bullet}_{A}(X) \;:=\; \frac{\ker\del\delbar}{\imm\del+\imm\delbar} \;.$$

\medskip

The identity induces the natural maps of (bi-)graded $\C$-vector spaces
$$
\xymatrix{
 & H^{\bullet,\bullet}_{BC}(X) \ar[d]\ar[ld]\ar[rd] & \\
 H^{\bullet,\bullet}_{\del}(X) \ar[rd] & H^{\bullet}_{dR}(X;\C) \ar[d] & H^{\bullet,\bullet}_{\delbar}(X) \ar[ld] \\
 & H^{\bullet,\bullet}_{A}(X) &
}
$$
In general, the maps above are neither injective nor surjective. The compact complex manifold $X$ is said to {\em satisfy the $\del\delbar$-Lemma} if and only if
$$ \ker\del\cap\ker\delbar\cap\imm\de \;=\; \imm\del\delbar \;, $$
that is, if and only if the map $H^{\bullet,\bullet}_{BC}(X)\to H^{\bullet}_{dR}(X;\C)$ is injective. This turns out to be equivalent to say that all the maps above are isomorphisms, see \cite[Remark 5.16]{deligne-griffiths-morgan-sullivan}. As already reminded, compact K\"ahler manifolds and, more in general, compact complex manifolds in class $\mathcal{C}$ of Fujiki, \cite{fujiki}, satisfy the $\del\delbar$-Lemma, see \cite[Corollary 5.23]{deligne-griffiths-morgan-sullivan}.

\medskip

There is a Hodge theory also for Bott-Chern and Aeppli cohomologies, see \cite{schweitzer}. More precisely, fixed a Hermitian metric on $X$, one has that
$$ H^{\bullet,\bullet}_{BC}(X) \;\simeq\; \ker\tilde\Delta_{BC} \qquad \text{ and }\qquad H^{\bullet,\bullet}_{A}(X) \;\simeq\; \ker\tilde\Delta_{A} \;,$$
where
$$ \tilde\Delta_{BC} \;:=\;
\left(\del\delbar\right)\left(\del\delbar\right)^*+\left(\del\delbar\right)^*\left(\del\delbar\right)+\left(\delbar^*\del\right)\left(\delbar^*\del\right)^*+\left(\delbar^*\del\right)^*\left(\delbar^*\del\right)+\delbar^*\delbar+\del^*\del $$
and
$$ \tilde\Delta_{A} \;:=\; \del\del^*+\delbar\delbar^*+\left(\del\delbar\right)^*\left(\del\delbar\right)+\left(\del\delbar\right)\left(\del\delbar\right)^*+\left(\delbar\del^*\right)^*\left(\delbar\del^*\right)+\left(\delbar\del^*\right)\left(\delbar\del^*\right)^* $$
are $4$-th order elliptic self-adjoint differential operators. In particular, one gets that
$$ \dim_\C H^{\bullet,\bullet}_{\sharp}(X) \;<\;+\infty \qquad \text{ for } \sharp\in\left\{\delbar,\,\del,\,BC,\,A\right\} \;.$$
By the definition of the Laplacians, it follows that
\begin{eqnarray*}
u \;\in\; \ker\tilde\Delta_{BC} \quad &\Leftrightarrow& \quad \del u \;=\; \delbar u \;=\; \left(\del\delbar\right)^* u \;=\; 0 \\[5pt]
 & \Leftrightarrow & \quad \del^* \left(*u\right) \;=\; \delbar^* \left(*u\right) \;=\; \del\delbar \left(*u\right) \;=\; 0 \quad \Leftrightarrow \quad *u \;\in\; \ker\tilde\Delta_A
\end{eqnarray*}
and hence the duality
$$ *\colon H^{p,q}_{BC}(X) \stackrel{\simeq}{\to} H^{n-q,n-p}_{A}(X) \;,$$
for every $p,q\in\N$.

\medskip

As a matter of notation, for every $p,q\in\N$, for every $k\in\N$ and for $\sharp\in\left\{\delbar,\,\del,\,BC,\,A\right\}$, we will denote
$$ h^{p,q}_{\sharp} \;:=\; \dim_\C H^{p,q}_{\sharp}(X) \qquad \text{ and }\qquad h^{k}_{\sharp} \;:=\; \sum_{p+q=k}h^{p,q}_{\sharp} \;, $$
and we will denote the Betti numbers by
$$ b_k \;:=\; \dim_\C H^{k}_{dR}(X;\C) \;.$$

Recall that, by conjugation and by the duality induced by the Hodge-$*$-operator associated to a given Hermitian metric, for every $p,q\in\N$ and for every $k\in\N$, one has the following equalities:
$$ h^{p,q}_{BC} \;=\; h^{q,p}_{BC} \;=\; h^{n-p,n-q}_{A} \;=\; h^{n-q,n-p}_{A} \quad \text{ and }\quad h^{p,q}_{\delbar} \;=\; h^{q,p}
_{\del} \;=\; h^{n-p,n-q}_{\delbar} \;=\; h^{n-q,n-p}_{\del} \;,$$
and therefore
$$ h^k_{BC} \;=\; h^{2n-k}_A \quad \text{ and } \quad h^k_{\delbar} \;=\; h^k_{\del} \;=\; h^{2n-k}_{\delbar} \;=\; h^{2n-k}_{\del} \;;$$
lastly, recall that the Hodge-$*$-operator (of any given Riemannian metric on $X$) yields, for every $k\in\N$, the equality
$$ b_k \;=\; b_{2n-k} \;.$$

\section{Proof of Theorems \ref{thm:frol-bc} and \ref{thm:caratterizzazione-bc-numbers}}\label{sec:main}

In this section, we prove the main results, stating a Fr\"olicher-type inequality for Bott-Chern and Aeppli cohomologies and giving therefore a characterization of compact complex manifolds satisfying the $\del\delbar$-Lemma in terms of the dimensions of their Bott-Chern cohomology groups.

\medskip

First of all, we need to recall two exact sequences from \cite{varouchas}. Let $X$ be a compact complex manifold of complex dimension $n$. Following J. Varouchas, one defines the finite-dimensional bi-graded vector spaces
$$ A^{\bullet,\bullet} \;:=\; \frac{\imm\delbar\cap\imm\del}{\imm\del\delbar} \;,\quad B^{\bullet,\bullet}\;:=\; \frac{\ker\delbar\cap\imm\del}{\imm\del\delbar} \;,\quad C^{\bullet,\bullet}\;:=\; \frac{\ker\del\delbar}{\ker\delbar+\imm\del} $$
and
$$ D^{\bullet,\bullet} \;:=\; \frac{\imm\delbar\cap\ker\del}{\imm\del\delbar} \;,\quad E^{\bullet,\bullet}\;:=\; \frac{\ker\del\delbar}{\ker\del+\imm\delbar} \;,\quad F^{\bullet,\bullet}\;:=\; \frac{\ker\del\delbar}{\ker\delbar+\ker\del} \;.$$
For every $p,q\in\N$ and $k\in\N$, we will denote
$$ a^{p,q}\;:=\; \dim_\C A^{p,q}\;,\quad\ldots\;,\quad f^{p,q}\;:=\; \dim_\C F^{p,q} $$
and
$$ a^k\;:=\; \sum_{p+q=k}a^{p,q}\;,\quad \ldots\;,\quad f^k\;:=\; \sum_{p+q=k}f^{p,q}\;.$$
One has the following exact sequences, see \cite[\S3.1]{varouchas}:
\begin{equation}\label{eq:succesatta-1}
0 \to A^{\bullet,\bullet} \to B^{\bullet,\bullet} \to H^{\bullet,\bullet}_{\delbar}(X) \to H^{\bullet,\bullet}_{A}(X) \to C^{\bullet,\bullet} \to 0
\end{equation}
and
\begin{equation}\label{eq:succesatta-2}
0 \to D^{\bullet,\bullet} \to H^{\bullet,\bullet}_{BC}(X) \to H^{\bullet,\bullet}_{\delbar}(X) \to E^{\bullet,\bullet} \to F^{\bullet,\bullet} \to 0 \;.
\end{equation}
Note also (see \cite[\S3.1]{varouchas})
that the conjugation and the maps $\delbar\colon C^{\bullet,\bullet}\stackrel{\simeq}{\to}D^{\bullet,\bullet+1}$ and $\del\colon E^{\bullet,\bullet}\stackrel{\simeq}{\to} B^{\bullet+1,\bullet}$ induce, for every $p,q\in\N$, the equalities
\begin{equation}\label{eq:apq=aqp}
a^{p,q}\;=\;a^{q,p}\;, \qquad f^{p,q} \;=\; f^{q,p}\;, \qquad d^{p,q}\;=\;b^{q,p}\;, \qquad e^{p,q}\;=\;c^{q,p}
\end{equation}
and
$$ c^{p,q}\;=\;d^{p,q+1}\;,\qquad e^{p,q}\;=\;b^{p+1,q}\;,$$
from which one gets, for every $k\in\N$, the equalities
$$ d^k\;=\;b^k\;,\qquad e^k\;=\;c^k\qquad \text{ and } \qquad c^k\;=\; d^{k+1}\;,\qquad e^k\;=\;b^{k+1}\;.$$

\begin{rem}
 Note that the argument used to prove the duality between Bott-Chern and Aeppli cohomology groups, see \cite{schweitzer}, can be applied to show also the dualities between $A^{\bullet,\bullet}$ and $F^{\bullet,\bullet}$ and between $C^{\bullet,\bullet}$ and $D^{\bullet,\bullet}$.
\end{rem}

\medskip

We can now state a Fr\"olicher-type inequality for Bott-Chern and Aeppli cohomologies.
While, on every compact complex manifold, one has the Fr\"olicher inequality $h^k_{\delbar}\geq b_k$ for every $k\in\N$, see \cite{frolicher}, this holds no more true for $h^k_{BC}$, as the following example shows.

\begin{ex}\label{ex:iwasawa}
Let $\mathbb{H}(3;\mathbb{C})$ be the $3$-dimensional complex \emph{Heisenberg group} defined by
$$
\mathbb{H}(3;\mathbb{C}) := \left\{
\left(
\begin{array}{ccc}
 1 & z^1 & z^3 \\
 0 &  1  & z^2 \\
 0 &  0  &  1
\end{array}
\right) \in \mathrm{GL}(3;\mathbb{C}) \st z^1,\,z^2,\,z^3 \in\C \right\}
\;.
$$
Define the {\em Iwasawa manifold} as the $3$-dimensional compact complex manifold given by the quotient
$$ \mathbb{I}_3 := \left. \mathbb{H}\left(3;\Z\left[\im\right]\right) \right\backslash \mathbb{H}(3;\C)\;, $$
where $\mathbb{H}\left(3;\Z\left[\im\right]\right):=\mathbb{H}(3;\C)\cap\GL\left(3;\Z\left[\im\right]\right)$.\\
The Kuranishi space of $\mathbb{I}_3$ is smooth and depends on $6$ effective parameters, see \cite{nakamura}. According to I. Nakamura's classification, the small deformations
of $ \mathbb{I}_3$ are divided into three classes, {\itshape (i)}, {\itshape (ii)} and {\itshape (iii)}, in terms of their Hodge numbers: such classes
are explicitly described by means of polynomial (in)equalities in the
parameters, see \cite[\S 3]{nakamura}.\\
The dimensions of the Bott-Chern and Aeppli cohomology groups for $\mathbb{I}_3$ are computed in \cite[Proposition 1.2]{schweitzer}, for the small deformations of $\mathbb{I}_3$ are computed in \cite[\S 5.3]{angella} (we refer to it for more details). It turns out that the Bott-Chern cohomology yields a finer classification of the Kuranishi space of $ \mathbb{I}_3$. More precisely, $h^{2,2}_{BC}$ assumes different values within class {\itshape (ii)}, respectively class {\itshape (iii)}, according to the rank of a certain matrix whose entries are related to
the complex structure equations with respect to a suitable co-frame (see \cite[\S 4.2]{angella}), whereas the numbers corresponding to class {\itshape (i)} coincide with those for $ \mathbb{I}_3$: this allows a further subdivision of  classes {\itshape (ii)} and {\itshape (iii)} into subclasses {\itshape (ii.a)}, {\itshape (ii.b)}, and {\itshape (iii.a)}, {\itshape (iii.b)}. For the sake of completeness, we list here these numbers.

\smallskip
\begin{small}
\begin{table}
\begin{center}
\begin{tabular}{c||*{3}{c}|*{3}{c}|*{3}{c}|*{3}{c}|*{3}{c}||}
\toprule
  {\bfseries classes} & $\mathbf{h^1_{\delbar}}$ & $\mathbf{h^1_{BC}}$ & $\mathbf{h^1_{A}}$ & $\mathbf{h^2_{\delbar}}$ & $\mathbf{h^2_{BC}}$ & $\mathbf{h^2_{A}}$ & $\mathbf{h^3_{\delbar}}$ & $\mathbf{h^3_{BC}}$ & $\mathbf{h^3_{A}}$ & $\mathbf{h^4_{\delbar}}$ & $\mathbf{h^4_{BC}}$ & $\mathbf{h^4_{A}}$ & $\mathbf{h^5_{\delbar}}$ & $\mathbf{h^5_{BC}}$ & $\mathbf{h^5_{A}}$\\
\midrule[0.02em]\midrule[0.02em]
{\itshape (i)} & 5 & 4 & 6 & 11 & 10 & 12 & 14 & 14 & 14 & 11 & 12 & 10 & 5 & 6 & 4\\
\midrule[0.02em]
{\itshape (ii.a)} & 4 & 4 & 6 & 9 & 8 & 11 & 12 & 14 & 14 & 9 & 11 & 8 & 4 & 6 & 4\\
{\itshape (ii.b)} & 4 & 4 & 6 & 9 & 8 & 10 & 12 & 14 & 14 & 9 & 10 & 8 & 4 & 6 & 4\\
\midrule[0.02em]
{\itshape (iii.a)} & 4 & 4 & 6 & 8 & 6 & 11 & 10 & 14 & 14 & 8 & 11 & 6 & 4 & 6 & 4\\
{\itshape (iii.b)} & 4 & 4 & 6 & 8 & 6 & 10 & 10 & 14 & 14 & 8 & 10 & 6 & 4 & 6 & 4\\
\midrule[0.02em]\midrule[0.02em]
 & \multicolumn{3}{c|}{$\mathbf{b_1=4}$} & \multicolumn{3}{c|}{$\mathbf{b_2=8}$} & \multicolumn{3}{c|}{$\mathbf{b_3=10}$} & \multicolumn{3}{c|}{$\mathbf{b_4=8}$} & \multicolumn{3}{c||}{$\mathbf{b_5=4}$} \\
\bottomrule
\end{tabular}
\end{center}
\caption{Dimensions of the cohomologies of the Iwasawa manifold and of its small deformations.}
\end{table}
\end{small}
\smallskip

\end{ex}

Nevertheless, we can prove the following.

\begin{thm}\label{thm:frol-bc}
 Let $X$ be a compact complex manifold of complex dimension $n$. Then, for every $p,q\in\N$, the following inequality holds:
\begin{equation}\label{eq:frol-bc-1}
h^{p,q}_{BC} + h^{p,q}_{A} \;\geq\; h^{p,q}_{\delbar}+h^{p,q}_\del \;.
\end{equation}
In particular, for every $k\in\N$, the following inequality holds:
\begin{equation}\label{eq:frol-bc-2}
h^k_{BC}+h^k_{A} \;\geq\; 2\,b_k \;,
\end{equation}
where $h^k_{BC}:=\sum_{p+q=k}\dim_\C H^{p,q}_{BC}(X)$ and $h^k_{A}:=\sum_{p+q=k}\dim_\C H^{p,q}_{A}(X)$.
\end{thm}

\begin{proof}
Fix $p,q\in\N$;  using the symmetries $h^{p,q}_A=h^{q,p}_A$ and $h^{p,q}_{\delbar}=h^{q,p}_{\del}$, the exact sequences \eqref{eq:succesatta-1} and \eqref{eq:succesatta-2} and the equalities \eqref{eq:apq=aqp}, we have
\begin{eqnarray*}
h^{p,q}_{BC}+h^{p,q}_{A} &=& h^{p,q}_{BC}+h^{q,p}_{A} \\[5pt]
&=& h^{p,q}_{\delbar} + h^{q,p}_{\delbar} + f^{p,q} + a^{q,p} + d^{p,q} - b^{q,p} - e^{p,q} + c^{q,p} \\[5pt]
&=& h^{p,q}_{\delbar} + h^{p,q}_{\del} + f^{p,q} + a^{p,q}  \\[5pt]
&\geq& h^{p,q}_{\delbar} + h^{p,q}_{\del} \;,
\end{eqnarray*}
which proves \eqref{eq:frol-bc-1}.\\
Now, fix $k\in\N$; summing over $\left(p,q\right)\in\N\times\N$ such that $p+q=k$, we get
\begin{eqnarray*}
 h^k_{BC} + h^k_{A} &=& \sum_{p+q=k} \left( h^{p,q}_{BC} + h^{p,q}_{A} \right) \\[5pt]
&\geq& \sum_{p+q=k} \left( h^{p,q}_{\delbar} + h^{p,q}_{\del}\right) \;=\; h^k_{\delbar} + h^k_{\del} \\[5pt]
&\geq& 2\, b_k \;,
\end{eqnarray*}
from which we get \eqref{eq:frol-bc-2}.
\end{proof}

\begin{rem}
 Note that small deformations of the Iwasawa manifold in Example \ref{ex:iwasawa} show that both the inequalities \eqref{eq:frol-bc-1} and \eqref{eq:frol-bc-2} can be strict.
\end{rem}

\begin{rem}
 Note that we have actually proved that, for every $k\in\N$,
\begin{equation*}
 h^k_{BC}+h^k_{A} \;=\; 2\,h^k_{\delbar}+a^k+f^k \;.
\end{equation*}
\end{rem}

\medskip

We prove now that equality in \eqref{eq:frol-bc-2} holds for every $k\in\N$ if and only if the $\del\delbar$-Lemma holds; in particular, this gives a characterization of the validity of the $\del\delbar$-Lemma just in terms of $\left\{h^k_{BC}\right\}_{k\in\N}$.

\begin{thm}\label{thm:caratterizzazione-bc-numbers}
 Let $X$ be a compact complex manifold. The equality
$$ h^k_{BC}+h^k_{A} \;=\; 2\,b_k $$
in \eqref{eq:frol-bc-2} holds for every $k\in\N$ if and only if $X$ satisfies the $\del\delbar$-Lemma.
\end{thm}

\begin{proof}
 Obviously, if $X$ satisfies the $\del\delbar$-Lemma, then, for every $k\in\N$, one has
$$ h^k_{BC} \;=\; h^k_{A} \;=\; h^k_{\delbar} \;=\; b_k $$
and hence, in particular,
$$ h^k_{BC}+h^k_{A} \;=\; 2\,b_k \;.$$
 We split the proof of the converse into the following claims.
\smallskip
 \paragraph{\texttt{\bfseries Claim 1.}} {\itshape If $h^k_{BC}+h^k_{A}=2\,b_k$ holds for every $k\in\N$, then $E_1\simeq E_{\infty}$ and $a^k=0=f^k$ for every $k\in\N$.}\\
Since, for every $k\in\N$, we have
$$ 2\,b_k \;=\; h^k_{BC}+h^k_{A} \;=\; 2\,h^k_{\delbar}+a^k+f^k \;\geq\; 2\,b_k \;,$$
then $h^k_{\delbar}=b_k$ and $a^k=0=f^k$ for every $k\in\N$.
\smallskip
 \paragraph{\texttt{\bfseries Claim 2.}} {\itshape  Fix $k\in\N$. If $a^{k+1}:=\sum_{p+q=k+1}\dim_\C A^{p,q}=0$, then the natural map
$$
\bigoplus_{p+q=k}H^{p,q}_{BC}(X)\to H^{k}_{dR}(X;\C)
$$
is surjective.}\\
 Let $\mathfrak{a}=\left[\alpha\right]\in H^k_{dR}(X;\C)$. We have to prove that $\mathfrak{a}$ admits a representative whose pure-type components are $\de$-closed. Consider the pure-type decomposition of $\alpha$:
$$ \alpha \;=:\; \sum_{j=0}^{k} \left(-1\right)^j\, \alpha^{k-j,j} \;,$$
where $\alpha^{k-j,j}\in\wedge^{k-j,j}X$. Since $\de\alpha=0$, we get that
$$ \del\alpha^{k,0}=0\;,\qquad \delbar\alpha^{k-j,j}-\del\alpha^{k-j-1,j+1}=0\text{ for }j\in\{0,\ldots,k-1\}\;,\qquad \delbar\alpha^{0.k}=0 \;; $$
by the hypothesis $a^{k+1}=0$, for every $j\in\{0,\ldots,k-1\}$, we get that,
$$ \delbar\alpha^{k-j,j}\;=\;\del\alpha^{k-j-1,j+1}\;\in\;\left(\imm\delbar\cap\imm\del\right)\cap\wedge^{k-j, j+1}X \;=\; \imm\del\delbar\cap\wedge^{k-j, j+1}X $$
and hence there exists $\eta^{k-j-1,j}\in\wedge^{k-j-1,j}X$ such that
$$ \delbar\alpha^{k-j,j} \;=\; \del\delbar\eta^{k-j-1,j} \;=\; \del\alpha^{k-j-1,j+1} \;.$$
Define
$$ \eta \;:=\; \sum_{j=0}^{k-1}\left(-1\right)^j\,\eta^{k-j-1,j} \;\in\;\wedge^{k-1}(X;\C) \;.$$
The claim follows noting that
\begin{eqnarray*}
\mathfrak{a} &=& \left[\alpha\right] \;=\; \left[\alpha+\de\eta\right] \\[5pt]
&=& \left[\left(\alpha^{k,0}+\del\eta^{k-1,0}\right)+\sum_{j=1}^{k-1}
\left(-1\right)^{j}\,\left(\alpha^{k-j,j}+\del\eta^{k-j-1,j}-\delbar\eta^{k-j,j-1}\right) \right.\\[5pt]
&{}&+ \left.\left(-1\right)^k\,\left(\alpha^{0,k}-\delbar\eta^{0,k-1}\right)\right] \\[5pt]
&=& \left[\alpha^{k,0}+\del\eta^{k-1,0}\right]+\sum_{j=1}^{k-1}\left(-1\right)^{j}\,\left[\alpha^{k-j,j}+\del\eta^{k-j-1,j}-\delbar\eta^{k-j,j-1}\right]\\[5pt]
&{}&+ \left(-1\right)^k\,\left[\alpha^{0,k}-\delbar\eta^{0,k-1}\right] \;,
\end{eqnarray*}
that is, each of the pure-type components of $\alpha+\de\eta$ is both $\del$-closed and $\delbar$-closed.
\smallskip
 \paragraph{\texttt{\bfseries Claim 3.}} {\itshape If $h^k_{BC}\geq b_k$ and $h^k_{BC}+h^k_{A}=2\,b_k$ for every $k\in\N$, then $h^k_{BC}=b_k$ for every $k\in\N$.}\\
If $n$ is the complex dimension of $X$, then, for every $k\in\N$, we have
$$ b_k \;\leq\; h^k_{BC} \;=\; h^{2n-k}_{A} \;=\; 2\,b_{2n-k}-h^{2n-k}_{BC} \;\leq\; b_{2n-k} \;=\; b_k$$
and hence $h^k_{BC}=b_k$ for every $k\in\N$.
\smallskip

Now, by \texttt{Claim 1}, we get that $a^k=0$ for each $k\in\N$; hence, by \texttt{Claim 2}, for every $k\in\N$ the map
$$ \bigoplus_{p+q=k}H^{p,q}_{BC}(X)\to H^k_{dR}(X;\C) $$
is surjective and hence, in particular, $h^k_{BC}\geq b_k$. By \texttt{Claim 3} we get therefore that $h^k_{BC}=b_k$ for every $k\in\N$. Hence, the natural map $H^{\bullet,\bullet}_{BC}(X)\to H^{\bullet}_{dR}(X;\C)$ is actually an isomorphism, which is equivalent to say that $X$ satisfies the $\del\delbar$-Lemma.
\end{proof}

\begin{rem}
We note that, using the exact sequences \eqref{eq:succesatta-2} and \eqref{eq:succesatta-1}, one can prove that, on a compact complex manifold $X$ and for every $k\in\N$,
\begin{eqnarray*}
e^k &=& \left(h^{k}_{\delbar}-h^{k}_{BC}\right)+f^k+c^{k-1} \\[5pt]
&=& \left(h^{k}_{\delbar}-h^{k}_{BC}\right)-\left(h^{k-1}_{\delbar}-h^{k-1}_{A}\right) + f^k-a^{k-1}+e^{k-2} \;.
\end{eqnarray*}
\end{rem}

\begin{rem}
 Note that $E_1\simeq E_\infty$ is not sufficient to have the equality $h^k_{BC}+h^k_{A}=2\,b_k$ for every $k\in\N$: a counter-example is provided by small deformations of the Iwasawa manifold, see Example \ref{ex:iwasawa}.
\end{rem}

\medskip

Using Theorem \ref{thm:caratterizzazione-bc-numbers}, we get another proof of the following result (see, e.g., \cite{voisin, wu}).

\begin{cor}\label{cor:stab}
 Satisfying the $\del\delbar$-Lemma is a stable property under small deformations of the complex structure.
\end{cor}

\begin{proof}
 Let $\left\{X_t\right\}_t$ be a complex-analytic family of compact complex manifolds. Since, for every $k\in\N$, the dimensions $h^{k}_{BC}(X_t)$ are upper-semi-continuous functions at $t$ (see, e.g., \cite{schweitzer}), while the dimensions $b_k(X_t)$ are constant in $t$, one gets that, if $X_{t_0}$ satisfies the equality $h^k_{BC}\left(X_{t_0}\right)+h^k_{A}\left(X_{t_0}\right)=2\,b_k\left(X_{t_0}\right)$ for every $k\in\N$, the same holds true for $X_t$ with $t$ near $t_0$.
\end{proof}

It could be interesting to construct a compact complex manifold (of any complex dimension greater or equal to $3$) such that $E_1\simeq E_{\infty}$ and $h^{p,q}_{\delbar}=h^{p,q}_{\del}$ for every $p,q\in\N$ but for which the $\del\delbar$-Lemma does not hold.\\
A compact complex manifold $X$ whose double complex $\left(\wedge^{\bullet,\bullet}X,\,\del,\,\delbar\right)$ has the form in Figure \ref{fig:conj} provides such an example.
\vfill\eject

\smallskip
\begin{figure}[h]
 \centering
 \includegraphics{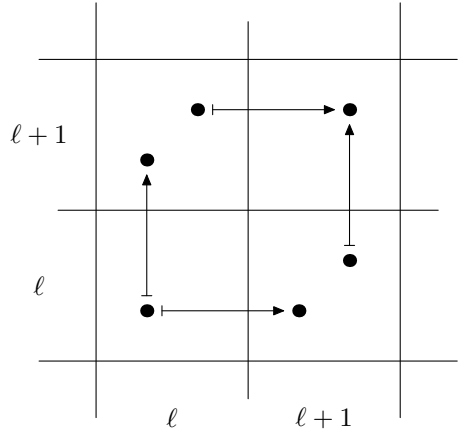}
 \caption{An abstract example}
 \label{fig:conj}
\end{figure}
\smallskip

\bigskip

\noindent{\sl Acknowledgments.}
The authors would like to thank warmly Jean-Pierre Demailly, for his support and encouragement and for his hospitality at Institut Fourier in Grenoble, and Greg Kuperberg, for many interesting conversations at Institut Fourier.
The authors would like also to thank Lucia Alessandrini for her interest and for many useful conversations.


\begin{thebibliography}{10}

\bibitem{angella}
D. Angella, The cohomologies of the Iwasawa manifold and of its small deformations, to appear in {\em J. Geom. Anal.}, published online doi: 10.1007/s12220-011-9291-z.

\bibitem{bismut}
J.-M. Bismut, Hypoelliptic Laplacian and Bott-Chern cohomology, preprint (Orsay) (2011).

\bibitem{deligne-griffiths-morgan-sullivan}
P. Deligne, Ph. Griffiths, J. Morgan, D.  Sullivan, Real homotopy theory of K\"ahler manifolds, {\em Invent. Math.} \textbf{29} (1975), no.~3, 245--274.

\bibitem{frolicher}
A. Fr\"{o}licher, Relations between the cohomology groups of Dolbeault and topological invariants, {\em Proc. Nat. Acad. Sci. U.S.A.} \textbf{41} (1955), no.~9, 641--644.

\bibitem{fujiki}
A. Fujiki, On automorphism groups of compact K\"ahler manifolds, {\em Invent. Math.} \textbf{44} (1978), no.~3, 225--258.

\bibitem{nakamura}
I. Nakamura, Complex parallelisable manifolds and their small deformations, {\em J. Differential Geom.} \textbf{10} (1975), no.~1, 85--112.

\bibitem{schweitzer}
M. Schweitzer, Autour de la cohomologie de Bott-Chern, \texttt{arXiv:0709.3528v1 [math.AG]}.

\bibitem{tseng-yau}
L.-S. Tseng, S.-T. Yau, Generalized Cohomologies and Supersymmetry, \texttt{arXiv:1111.6968v1 [hep-th]}.

\bibitem{varouchas}
J. Varouchas, Proprietés cohomologiques d'une classe de variétés analytiques complexes compactes, {\em Séminaire d'analyse P. Lelong-P. Dolbeault-H. Skoda, années 1983/1984}, 233--243, Lecture Notes in Math., \textbf{1198}, Springer, Berlin, 1986.

\bibitem{voisin}
C. Voisin, {\em Théorie de Hodge et géométrie algébrique complexe}, Cours Spécialisés, \textbf{10}, Société Mathématique de France, Paris, 2002.

\bibitem{wu}
C.-C. Wu, On the geometry of superstrings with torsion, Thesis (Ph.D.)Harvard University, Proquest LLC, Ann Arbor, MI, 2006.
\end{thebibliography}
\end{document}